\documentclass[12pt]{article}
\usepackage{amsmath,amssymb,amsthm}
\usepackage{comment}
\usepackage{hyperref}
\usepackage{booktabs}

\newtheorem{theorem}{Theorem}[section]
\newtheorem{corollary}[theorem]{Corollary}
\newtheorem{lemma}[theorem]{Lemma}
\newtheorem{conjecture}[theorem]{Conjecture}
\theoremstyle{remark}

\newenvironment{repeatconjecture}[1]
{%
  \begingroup
  \renewcommand{\thetheorem}{\ref{#1}}%
  \begin{conjecture}
}
{%
  \end{conjecture}
  \endgroup
}

\begin{document}

\title{A simple proof of the Atkin-O’Brien partition congruence conjecture for powers of 13}
\author{Frank Garvan, Zhumagali Shomanov}
\date{}
\maketitle

\section*{Abstract}
In 1967, Atkin and O’Brien conjectured congruences for the partition function involving Hecke operators modulo powers of 13. In this paper, we provide a simple proof of this conjecture.

\section{Introduction}

The partition function $p(n)$ counts the number of distinct ways of expressing a positive integer $n$ as a sum of positive integers, without considering the order of the summands. For example, the partitions of $4$ are:

$$
4, \quad 3+1, \quad 2+2, \quad 2+1+1, \quad 1+1+1+1,
$$
so $p(4) = 5$.

The generating function for the partition function is an infinite product given by:

\begin{equation}
\label{eq:gen_func}
\sum_{n=0}^{\infty} p(n) q^n = \prod_{k=1}^{\infty} \frac{1}{1 - q^{k}},
\end{equation}
where $|q| < 1$.

The $q$-Pochhammer symbol is defined as:

\begin{equation}
(a;q)_n = \prod_{k=0}^{n-1} (1 - a q^{k}),
\end{equation}
and for $n = \infty$:
\begin{equation}
(a;q)_\infty = \prod_{k=0}^{\infty} (1 - a q^{k}).
\end{equation}
Using this notation, the generating function for the partition function can be concisely written as:

\begin{equation}
\label{eq:gen_func_pochhammer}
\sum_{n=0}^{\infty} p(n) q^{n} = \frac{1}{(q; q)_\infty} = \frac{1}{f(q)}.
\end{equation}

The Dedekind eta function is a modular form defined in the upper half-plane $\mathbb{H}$ by:

\begin{equation}
\label{eq:dedekind_eta}
\eta(\tau) = e^{\pi i \tau /12} \prod_{n=1}^{\infty} \left(1 - e^{2\pi i n \tau}\right),
\end{equation}
where $\tau \in \mathbb{H}$. Setting $q = e^{2\pi i \tau}$, we have $|q|<1$, so the product converges.
Using the $q$-Pochhammer symbol, the eta function can be expressed as:

\begin{equation}
\eta(\tau) = q^{1/24} (q; q)_\infty,
\end{equation}
where $q^{1/24} = e^{\pi i \tau /12}$. Therefore, the reciprocal of the eta function is:

\begin{equation}
\eta(\tau)^{-1} = q^{-1/24} \frac{1}{(q; q)_\infty},
\end{equation}
which is a weakly holomorphic modular form on $M_{-1/2}^!(1,\nu_\eta^{-1})$, where $\nu_\eta$ is the eta multiplier.
From equation \eqref{eq:gen_func_pochhammer}, we see that:

\begin{equation}
\frac{1}{(q; q)_\infty} = \sum_{n=0}^{\infty} p(n) q^{n}.
\end{equation}
Thus, the generating function of $p(n)$ is related to the Dedekind eta function by:

\begin{equation}
\eta(\tau)^{-1} = q^{-1/24} \sum_{n=0}^{\infty} p(n) q^{n}=\sum_{n=-1}^\infty p\left(\cfrac{n+1}{24}\right)q^{n/24},
\end{equation}
where $p(n)=0$ if $n$ is not a positive integer.

In this paper we will prove a conjecture stated by Atkin and O'Brien in \cite{AtkinOBrien1967}. We will follow their work and let
\[
P(N) =
\begin{cases}
p(n) & \text{if } N = 24n - 1, \\
0 & \text{if } N < -1 \text{ or } N \not\equiv -1 \pmod{24} \text{ or } N \text{ is nonintegral}.
\end{cases}
\]
In their work, they proved
\begin{theorem}
\label{AOBthm}
For all $\alpha\ge1$ there exists an integral constant $K_\alpha$ not divisible by 13 such that for all $N$
\begin{equation}
P(13^{\alpha+2}N)\equiv K_\alpha P(13^\alpha N)\pmod{13^\alpha}.
\end{equation}
\end{theorem}
\noindent They also made the following conjecture
\begin{conjecture}\label{AOBconj}
Let \(\alpha \geq 1\), and \(p \neq 13\) be a prime \( \geq 5\). Then there exists a constant \(k = k(p, \alpha)\) such that for all \(N\),
\[
P(p^2 \cdot 13^\alpha N) - \left\{ k - \left(\frac{-3 \cdot 13^\alpha N}{p}\right)p^{-2} \right\} P(13^\alpha N) + p^{-3} P\left(\frac{13^\alpha N}{p^2}\right) \equiv 0 \pmod{13^\alpha},
\]
where \(\left(\frac{a}{b}\right)\) is the Jacobi symbol.
\end{conjecture}
\noindent {\bf Note}: Atkin and O'Brien proved the conjecture for $\alpha=1, 2$.

An immediate corollary is
\begin{corollary}
\label{AOBcong}
For $\alpha\ge1$ and for primes $p\ge5$, $p\neq13$, there is a constant $k=k(p,\alpha)$ such that for all n coprime to p
\[
p\left(\cfrac{13^\alpha p^3n+1}{24}\right)\equiv k\, p\left(\cfrac{13^\alpha p n+1}{24}\right)\pmod{13^\alpha}.
\]
\end{corollary}

In \cite{Folsom}, the authors proved results similar to Theorem \ref{AOBthm}, Conjecture \ref{AOBconj}, and Corollary \ref{AOBcong} for powers of all primes $\ell$ such that $5\le \ell\le31$. Namely,
\begin{theorem}\label{FKOcong}
Suppose that $5 \le \ell \le 31$ is a prime, and that $\alpha \ge 1$. If $b_1 \equiv b_2 \pmod{2}$ are integers with $b_2 > b_1 > b_\ell(\alpha)$,
then there is an integer $A_\ell\bigl(b_1,b_2,\alpha\bigr)$ such that for every nonnegative integer $n$ we have
\[
p\!\left(\cfrac{\ell^{b_2}n+1}{24}\right)
\;\equiv\;
A_\ell\bigl(b_1,b_2,\alpha\bigr)\,
p\!\left(\frac{\ell^{b_1}n+1}{24}\right)
\pmod{\ell^\alpha}
\]
where 
\[
b_\ell(\alpha) := 2 \left( \left\lfloor \frac{\ell - 1}{12} \right\rfloor + 2 \right) \alpha - 3.
\]
\end{theorem}
\begin{theorem}
If $5 \leqslant \ell \leqslant 31$ and $\alpha \geqslant 1$, then for $b \geqslant b_\ell(\alpha)$ we have that $P_\ell(b; 24z) \pmod{\ell^\alpha}$ is an eigenform of all of the weight $k_\ell(\alpha) - \frac{1}{2}$ Hecke operators on $\Gamma_0(576)$. Here
\[
P_\ell(b;z):=\sum_{n=0}^\infty p\left(\cfrac{\ell^b n+1}{24}\right)q^{n/24}
\]
\end{theorem}
\begin{corollary}
Suppose that $5 \leqslant \ell \leqslant 31$ and that $\alpha \geqslant 1$. If $b \geqslant b_\ell(\alpha)$, then for every prime $p \geqslant 5$ there is an integer $\lambda_\ell(\alpha, p)$ such that for all $n$ coprime to $p$ we have
\begin{equation}
p\left(\frac{\ell^b n p^3 + 1}{24}\right) \equiv \lambda_\ell(\alpha, p) p\left(\frac{\ell^b n p + 1}{24}\right) \pmod{\ell^\alpha}.
\end{equation}
\end{corollary}
In \cite{Boylan}, Boylan and Webb showed that, in Theorem \ref{FKOcong}, we can take $b_\ell(\alpha)=2\alpha-1.$ They prove their results by showing that the generating functions 
\[
P_\ell(b;z):=\sum_{n=0}^\infty p\left(\cfrac{\ell^b n+1}{24}\right)q^{n/24}
\]
lie in a certain $\mathbb{Z}/\ell^m\mathbb{Z}$ module with finite rank, and for certain submodules the dimension is 1 when $5\le\ell\le31$.

In our proof of Conjecture \ref{AOBconj}, we mostly follow the proof of Atkin and O'Brien for Theorem \ref{AOBthm} combined with the work of Newman \cite{Newman1962}.

\section{Preliminaries}

The full modular group is defined as
\[
\mathrm{SL}_2(\mathbb{Z}) = \left\{ \begin{pmatrix} a & b \\ c & d \end{pmatrix} : a,b,c,d \in \mathbb{Z},\; ad-bc = 1 \right\}.
\]
This group acts on the upper half-plane
\[
\mathbb{H} = \{ z \in \mathbb{C} : \Im(z) > 0 \}
\]
by linear fractional (Möbius) transformations:
\[
\gamma z = \frac{az+b}{cz+d}, \quad \text{for } \gamma = \begin{pmatrix} a & b \\ c & d \end{pmatrix} \in \mathrm{SL}_2(\mathbb{Z}).
\]
We will consider the subgroup $\Gamma_0(N)$ which is defined as 
\[
\Gamma_0(N) = \left\{ \begin{pmatrix} a & b \\ c & d \end{pmatrix} \in \mathrm{SL}_2(\mathbb{Z}) : c \equiv 0 \pmod{N} \right\}
\]

A \emph{modular function} on $\Gamma_0(N)$ is a meromorphic function $f : \mathbb{H} \to \mathbb{C}$ satisfying:
\begin{enumerate}
    \item For all $\gamma = \begin{pmatrix} a & b \\ c & d \end{pmatrix}\in \Gamma_0(N)$,
    \[
    f(\gamma z) = f(z).
    \]
    
    \item The function $f$ is meromorphic on the upper half-plane.
    
    \item It might have poles at the cusps of $\Gamma_0(N)$.
\end{enumerate}
A central role in our study is played by the Dedekind eta function, defined for $z\in\mathbb{H}$ by
\[
\eta(z) = q^{1/24} \prod_{n=1}^{\infty} (1-q^n), \quad \text{with } q=e^{2\pi i z}.
\]
It is well known that $\eta(z)$ is holomorphic and non-vanishing on $\mathbb{H}$. Moreover, $\eta(z)$ satisfies the transformation law
\[
\eta(\gamma z) = \nu_\eta(\gamma) \, (cz+d)^{1/2} \, \eta(z),
\]
for all $\gamma = \begin{pmatrix} a & b \\ c & d \end{pmatrix}\in \mathrm{SL}_2(\mathbb{Z})$, where
\[
\nu_\eta(\gamma) = \begin{cases}
\left(\frac{d}{c}\right) e^{2\pi i((a+d)c - bd(c^2-1) - 3c)/24}, & \text{if } c \text{ is odd}, \\
\left(\frac{c}{d}\right) e^{2\pi i((a+d)c - bd(c^2-1) + 3d - 3 - 3cd)/24}, & \text{if } c \text{ is even},
\end{cases}
\]

We say that $f\in M_{\lambda+1/2}^!(N,\psi \nu_\eta^r)$, where $\lambda\in \mathbb{Z}$, $N\in\mathbb{N}$, and $\psi$ is a Dirichlet character modulo $N$, if
\begin{enumerate}
    \item For all $\gamma = \begin{pmatrix} a & b \\ c & d \end{pmatrix} \in \Gamma_0(N)$ and all $z \in \mathbb{H}$,
    \[
    f(\gamma z) = \psi(\gamma)\nu_\eta^r(\gamma) \, (cz+d)^{\lambda+1/2} f(z).
    \]
    
    \item The function $f$ is holomorphic on $\mathbb{H}$.
    
    \item It might have poles at the cusps of $\Gamma_0(N)$.
\end{enumerate}

The Hecke operators for $f\in M_{\lambda+1/2}^!(N,\psi \nu_\eta^r)$, $(r,6)=1$, and primes $p$ are explicitly defined as (\cite{Ahlgren1})
\begin{multline}
F\,|\,T_{p^2} = \sum_{n \equiv r(24)} \left(a(p^2n) + \left(\frac{-1}{p}\right)^{\frac{r-1}{2}} \left(\frac{12n}{p}\right)\psi(p)p^{\lambda-1}a(n)\right. \\ 
\left.+ \psi^2(p)p^{2\lambda-1}a\left(\frac{n}{p^2}\right)\right)q^{\frac{n}{24}}
\end{multline}

The Atkin $U_p$ operator is defined as
\begin{equation}
f|U_p=\cfrac1{p}\sum_{j=0}^{p-1}f\left(\cfrac{\tau+24j}{p}\right).
\end{equation}

\section{Results of Newman}

Here, we will summarize the results of Newman (\cite{Newman1962}) needed to prove the conjecture.

For distinct primes $l$ and $p$, $p\ge5$, and integers of opposite parity $r$ and $s$, let
\[
B(\tau) = \eta^r(\tau)\,\eta^s(l\,\tau),\qquad B^*(\tau)=\eta^s(\tau)\eta^r(l\tau)
\]
and
\[
\quad\;\; \Delta 
= 
\cfrac{(r + ls)(p^2-1)}{24},
\qquad
\Delta^* 
= 
\cfrac{(s + lr)(p^2-1)}{24}.
\]
Recall that the Fricke involution sends $\tau$ to $-1/l\tau$. Then Newman showed the following

\begin{lemma}[Newman]
Let 
\[
G(\tau)=\cfrac{T_{p^2}(B(\tau))}{B(\tau)}, \qquad G^*(\tau)=\cfrac{T_{p^2}(B^*(\tau))}{B^*(\tau)}.
\]
Then $G(\tau)$ and $G^*(\tau)$ are weakly holomorphic modular functions on $\Gamma_0(l)$, and 
\[
G\left(-\cfrac{1}{l\tau}\right)=G^*(\tau)
\]
\end{lemma}

The main result of Newman's paper is
\begin{theorem}[Newman]
At \(\tau = i\infty\), \(G(\tau)\) has a pole of order \(\lfloor \Delta/p^2\rfloor\) at most if \(\Delta \ge 0\), and a pole of order \(-\Delta\) if \(\Delta < 0\). At \(\tau = 0\), \(G(\tau)\) has a pole of order \(\lfloor \Delta^* / p^2\rfloor\) at most if \(\Delta^* \ge 0\), and a pole of order \(-\Delta^*\) if \(\Delta^* < 0\).
\end{theorem}

\section{Proof of the Conjecture}
Recall that
\[
P(N) =
\begin{cases}
p(n) & \text{if } N = 24n - 1, \\
0 & \text{if } N < -1 \text{ or } N \not\equiv -1 \pmod{24} \text{ or } N \text{ is nonintegral}.
\end{cases}
\]
Define
\begin{align}
\varphi(\tau) &= \frac{\eta(169\tau)}{\eta(\tau)} \label{eq:phi}\\
g(\tau) &= \left( \frac{\eta(13\tau)}{\eta(\tau)} \right)^2 , \label{eq:g}
\end{align}
and
\begin{equation}
L_1 = \varphi | U_{13} = \eta(13\tau)\,\left(\cfrac{1}{\eta(\tau)}\Bigr\rvert U_{13}\right)=\eta(13\tau) \sum_{N=0}^\infty P\bigl(13(24N + 11)\bigr) q^{N+\frac{11}{24}}. \label{eq:L1}
\end{equation}
In \cite{AtkinOBrien1967}, Atkin and O'Brien showed that:
\begin{equation}
L_1 = \sum_{r=1}^7 k_{1r} \, g(\tau)^r, \label{eq:L1_sum}
\end{equation}
giving
\begin{equation}
\label{eq:L1h}
\cfrac1{\eta(\tau)}\Bigr|U_{13}=\sum_{N=0}^{\infty} P\bigl(13(24N+11)\bigr)\,q^{N+\frac{11}{24}}
\;=\;
\sum_{r=1}^7
  k_{1r} \,\frac{\eta(13\tau)^{2r-1}}{\eta(\tau)^{2r}}.
\end{equation}
The valuations of $k_{1r}$ are given in the following table
$$
\begin{tabular}{c *{7}{@{\hspace{1.5em}} c}}
\toprule
$r$ & 1 & 2 & 3 & 4 & 5 & 6 & 7 \\ \midrule
$\pi(k_{1r})$ & 0 & 1 & 2 & 3 & 4 & 5 & 5 \\
\bottomrule
\end{tabular}
$$
Here $\pi(a)$, for integral $a$, is defined as:
\[
13^{\pi(a)} \mid a,\qquad 13^{\pi(a)+1} \nmid a.
\]
We clearly have:
\begin{align}
\pi(ab) &= \pi(a) + \pi(b), \\
\pi(a + b) &\geq \min(\pi(a), \pi(b)),
\end{align}
with equality unless $\pi(a) = \pi(b)$.

By (\cite{Ahlgren2}), 
\[\cfrac{1}{\eta(\tau)}\Bigr|U_{13}\in M_{-\frac12}\left(13,\left(\cfrac{\cdot}{13}\right)\nu_\eta^{11}\right).\]
Direct calculation shows that 
\[
\frac{\eta(13\tau)^{2r-1}}{\eta(\tau)^{2r}}\in M_{-\frac12}\left(13,\left(\cfrac{\cdot}{13}\right)\nu_\eta^{11}\right)
\]
too. Therefore, we can apply $T_{p^2}$ to both sides of (\ref{eq:L1h})
\begin{equation}
\cfrac1{\eta(\tau)}\Bigr|U_{13}\Bigr|T_{p^2}
\;=\;
\sum_{r=1}^7
  k_{1r} \,\left(\frac{\eta(13\tau)^{2r-1}}{\eta(\tau)^{2r}}\Bigr|T_{p^2}\right).
\end{equation}

Next, let
\begin{align*}
\omega_r &= \frac{\eta(13\tau)^{2r-1}}{\eta(\tau)^{2r}}, &
\omega^*_r &= \frac{\eta(\tau)^{2r-1}}{\eta(13\tau)^{2r}} , \\[2ex]
\Delta_r &= \frac{(24r - 13)(p^2 - 1)}{24}, &
\Delta_r^* &= -\frac{(24r + 1)(p^2 - 1)}{24}.
\end{align*}
and
\begin{align}
G_r(\tau) &= \cfrac{\omega_r|T_{p^2}}{\omega_r}, \label{eq:newhecke1}\\
G_r^*(\tau) &= \cfrac{\omega^*_r|T_{p^2}}{\omega^*_r}.
\end{align}
Then, using Newman's results, $G_r(\tau)$ is a weakly holomorphic modular function on $\Gamma_0(13)$, $G_r^*$ has a pole of order $-\Delta_r^*$ at $i\infty$, $G_r$ has a pole of order $\left\lfloor\Delta_r/p^2\right\rfloor$ at $i\infty$, and
\begin{equation}
\label{eq:inv}
G_r^*\left(-\frac{1}{13\tau}\right) = G_r(\tau).
\end{equation}
This means that
\[
G_r^*(\tau) \;-\; \sum_{i=1}^{-\Delta_r^*}\delta_{r,i}\,g(\tau)^{-\,i}
\]
has no poles at $i\infty$ ($\delta_{r,i}\in\mathbb{Z}$). Using (\ref{eq:inv}) and transformation properties of $\eta(\tau)$, we can rewrite that as
\[
G_r^*\left(-\frac{1}{13\tau}\right)
  \;-\;
  \sum_{i=1}^{-\Delta_r^*} \delta_{r,i}\,g\left(-\frac{1}{13\tau}\right)^{-\,i}
\;=\;
G_r(\tau)
  \;-\;
  \sum_{i=1}^{-\Delta_r^*} \delta_{r,i}\,13^i\,g(\tau)^i,
\]
which does not have poles at 0. Since $G_r$ has a pole of order $\lfloor\Delta_r/p^2\rfloor$ at $i\infty$, and $\lfloor\Delta_r/p^2\rfloor=r-1$ for $1\le r\le 7$, we have 
\[
G_r(\tau)
 \;-\;
 \sum_{i=1}^{-\Delta_r^*}
   \delta_{r,i}\,13^i\,g(\tau)^i
 \;-\;
 \sum_{i=1}^{r-1}
   \delta_{r,-i}\,g(\tau)^{-\,i}
\]
has no poles and must be a constant for $1 \le r \le 7$.
Hence in general one writes:
\[
G_r(\tau)
 \;=\;
 \sum_{i=1}^{r-1} \delta_{r,-i}\,g(\tau)^{-i}
  \;+\;
 \delta_{r,\,0}
  \;+\;
 \sum_{i=1}^{-\Delta_r^*} \delta_{r,i}\,13^i\,g(\tau)^i.
\]
Using (\ref{eq:newhecke1}), we will get
\begin{equation}
\omega_r|T_{p^2}=\omega_r(\tau)\left(\sum_{i=1}^{r-1} \delta_{r,-i}\,g(\tau)^{-i}
  \;+\;
 \delta_{r,\,0}
  \;+\;
 \sum_{i=1}^{-\Delta_r^*} \delta_{r,i}\,13^i\,g(\tau)^i\right),
\end{equation}
or, equivalently,
\begin{equation}
\omega_r|T_{p^2}=\sum_{i=1}^{-\Delta_r^*+r} h_{r,i}
   \,\omega_i.
\end{equation}
One also sees conditions on \(\pi(h_{r,i})\):
\[
\pi\bigl(h_{r,i}\bigr) \;\ge\; 0
  \quad \text{for }1\le i \le r,
\quad
\pi\bigl(h_{r,i}\bigr) \;\ge\; i-r
  \quad \text{for } i>r.
\]
Thus,
\begin{equation}
\cfrac1{\eta(\tau)}\Bigr|U_{13}\Bigr|T_{p^2}
\;=\;
\sum_{r=1}^7
  k_{1r} \,\left(\frac{\eta(13\tau)^{2r-1}}{\eta(\tau)^{2r}}\Bigr|T_{p^2}\right)=\sum_{r=1}^7k_{1r}\sum_{j=1}^{-\Delta_r^*+r} h_{r,j}
   \,\omega_j,
\end{equation}
or, after multiplying by $\eta(13\tau)$ and simplifying, we will get
\begin{align*}
& \eta(13\tau)\left(\cfrac1{\eta(\tau)}\Bigr|U_{13}\Bigr|T_{p^2}\right)=\eta(13\tau)\sum_{N=0}^\infty \Biggr(p^3 P(13 p^2 (24N + 11))+ \\[2mm]
& + p \left( \frac{-3}{p}\right)\left( \frac{13 (24N + 11)}{p} \right) P(13 (24N + 11))+ \\
& + P \left(\frac{13 (24N+11)}{p^2} \right)\Biggr) q^{N+\frac{11}{24}} \ = \sum_{r=1}^7 k_{1r} \sum_{j=1}^{-\Delta_r^*+r} h_{r,j} g^j = \sum_{j \geq 1} u_{1,j} g^j
\end{align*}

\noindent Let
\begin{align}
H_{2\alpha-1}&=\eta(13\tau) \left(\cfrac1{\eta(\tau)}\Bigr|U_{13}^{2\alpha-1}\Bigr|T_{p^2}\right)=\eta(13\tau)\sum_{N=0}^\infty \Biggr(p^3 P(13^{2\alpha-1} p^2 (24N + 11))\nonumber \\ 
&+ p \left( \frac{-3}{p}\right)\left( \frac{13^{2\alpha-1} (24N + 11)}{p} \right) P(13^{2\alpha-1} (24N + 11)) \\
&+ P \left(\frac{13^{2\alpha-1} (24N+11)}{p^2} \right)\Biggr) q^{N+\frac{11}{24}}\nonumber \\
H_{2\alpha}&=\eta(\tau) \left(\cfrac1{\eta(\tau)}\Bigr|U_{13}^{2\alpha}\Bigr|T_{p^2}\right)=\eta(13\tau)\sum_{N=0}^\infty \Biggr(p^3 P(13^{2\alpha} p^2 (24N + 23))\nonumber \\ 
&+ p \left( \frac{-3}{p}\right)\left( \frac{13^{2\alpha} (24N + 23)}{p} \right) P(13^{2\alpha} (24N + 23)) \\
&+ P \left(\frac{13^{2\alpha} (24N+23)}{p^2} \right)\Biggr) q^{N+\frac{23}{24}}\nonumber \\
L_{2\alpha-1}&=\eta(13\tau)\left(\cfrac1{\eta(\tau)}\Bigr|U_{13}^{2\alpha-1}\right)=\eta(13\tau)\sum_{N=0}^\infty P(13^{2\alpha-1}(24N+11))q^{N+\frac{11}{24}}\\
L_{2\alpha}&=\eta(\tau)\left(\cfrac1{\eta(\tau)}\Bigr|U_{13}^{2\alpha}\right)=\eta(\tau)\sum_{N=0}^\infty P(13^{2\alpha}(24N+23))q^{N+\frac{23}{24}}
\end{align}
In \cite{AtkinOBrien1967}, Atkin and O'Brien showed that
\begin{align}
 g^k|U_{13}  &= \sum_{r} c_{k,r}\, g^r, \\
 \varphi\, g^k\bigr|U_{13}  &= \sum_{r} d_{k,r}\, g^r,
\end{align}
and therefore
\begin{align}
L_{2\alpha - 1} &= \varphi L_{2\alpha - 2}\bigr|U_{13} = \sum\limits_{r} k_{\alpha,r} g^r \\
L_{2\alpha}     &=  L_{2\alpha - 1} |U_{13}       = \sum_r \ell_{\alpha,r} g^r.
\end{align}
Similarly, one can show that
\begin{align}
H_{2\alpha - 1} &= \varphi H_{2\alpha - 2}\bigr|U_{13} = \sum\limits_{r} u_{\alpha,r} g^r,\\
H_{2\alpha}     &= H_{2\alpha - 1}|U_{13}       = \sum\limits_{r} v_{\alpha,r} g^r.
\end{align}
Note that 
\begin{align}
l_{\alpha r}&=\sum_j k_{\alpha j}c_{jr}, \\
k_{\alpha+1, r}&=\sum_j l_{\alpha j}d_{jr}.
\end{align}

\begin{lemma}[Atkin - O'Brien]\label{AOB_val}
\begin{align*}
\pi\bigl(k_{\alpha,r}\bigr) &\ge \left\lfloor \dfrac{13r - 9}{14} \right\rfloor, &
\pi\bigl(\ell_{\alpha,r}\bigr) &\ge \left\lfloor \dfrac{13r - 2}{14} \right\rfloor, \\[2ex]
\pi\bigl(k_{\alpha,1}\bigr) & = 0, &
\pi\bigl(\ell_{\alpha,1}\bigr) & = 0, \\[2ex]
\pi\bigl(c_{k,r}\bigr) &\ge \left\lfloor \dfrac{13r - k - 1}{14} \right\rfloor, &
\pi\bigl(d_{k,r}\bigr) &\ge \left\lfloor \dfrac{13r - k - 8}{14} \right\rfloor.
\end{align*}
\end{lemma}

\begin{lemma}\label{AOB}
$$
\pi\bigl(u_{\alpha,r}\bigr) \ge \left\lfloor \dfrac{13r - 9}{14} \right\rfloor, \qquad
\pi\bigl(v_{\alpha,r}\bigr) \ge \left\lfloor \dfrac{13r - 2}{14} \right\rfloor.
$$
\end{lemma}

\begin{proof}
For $\alpha = 1$, we have
\begin{align*}
H_1 &= \sum_{j \ge 1} u_{1,j}\, g^j = \sum_{r=1}^7 k_{1,r} \sum_{j=1}^{-\Delta_r^*+r} h_{r,j}\, g^j \\
&= k_{1,1} \sum_{j} h_{1,j}\, g^j + k_{1,2} \sum_{j} h_{2,j}\, g^j + \cdots + k_{1,7} \sum_{j} h_{7,j}\, g^j.
\end{align*}
Recall that
\begin{align*}
\pi(k_{1,1}) &= 0, 
\pi(k_{1,2}) = 1, 
\pi(k_{1,3}) = 2, 
\pi(k_{1,4}) = 3, 
\pi(k_{1,5}) = 4,
\pi(k_{1,6}) = 5, \\
\pi(k_{1,7}) &= 5,
\end{align*}
and
$$
\pi(h_{r,j}) \ge 0 \quad \text{for } 1 \le j \le r, \qquad
\pi(h_{r,j}) \ge j - r \quad \text{for } j > r.
$$
Hence,
$$
\pi\bigl(k_{1,r}\, h_{r,j}\bigr) \ge j-1\ge \left\lfloor\frac{13j-9}{14}\right\rfloor,\qquad 0\le r \le 6,
$$
and
\[
\pi(k_{17}h_{7j}) \geq \begin{cases} 
5 & \text{for } j = 1,2,\ldots,7 \\
j-2 & \text{for } j \geq 8
\end{cases} \geq \left\lfloor\frac{13j-9}{14}\right\rfloor.
\]
Therefore,
$$
\pi(u_{1,j}) \ge \min_r\{\pi\bigl(k_{1,r}\, h_{r,j}\bigr)\} \ge \left\lfloor \dfrac{13j - 9}{14} \right\rfloor.
$$
Next, consider
\begin{multline*}
H_{2\alpha} = \sum_{j \ge 1} v_{\alpha,j}\, g^j = H_{2\alpha - 1}|U_{13} = \sum_{r \ge 1} u_{\alpha,r}\, \left(g^r\Bigr|U_{13}\right) \\
= \sum_{r \ge 1} u_{\alpha,r} \sum_{j \ge 1} c_{r,j}\, g^j = \sum_{j \ge 1} g^j \sum_{r \ge 1} u_{\alpha,r}\, c_{r,j}.
\end{multline*}
Using the induction hypothesis and Lemma~\ref{AOB_val}
$$
\pi(u_{\alpha,r}) \ge \left\lfloor \dfrac{13r - 9}{14} \right\rfloor, \qquad
\pi(c_{r,j}) \ge \left\lfloor \dfrac{13j - r - 1}{14} \right\rfloor,
$$
we get
\begin{align*}
\pi\left( \sum_{r \ge 1} u_{\alpha,r}\, c_{r,j} \right) &\ge \min_{r} \left\{ \pi(u_{\alpha,r}) + \pi(c_{r,j}) \right\} \\
&\ge \min_{r} \left\{ \left\lfloor \dfrac{13r - 9}{14} \right\rfloor + \left\lfloor \dfrac{13j - r - 1}{14} \right\rfloor \right\} = \left\lfloor \dfrac{13j - 2}{14} \right\rfloor,
\end{align*}
since the minimum occurs at either $r=1$ or $r=2$, but in fact at $r=1$.
Hence,
$$
\pi(v_{\alpha,r}) \ge \left\lfloor \dfrac{13r - 2}{14} \right\rfloor.
$$
Next,
\begin{multline*}
H_{2\alpha + 1} = \sum_{j \ge 1} u_{\alpha+1,j}\, g^j = \varphi H_{2\alpha}|U_{13} = \sum_{r \ge 1} v_{\alpha,r}\, \left(\varphi g^r \Bigr|U_{13}\right) \\
= \sum_{r \ge 1} v_{\alpha,r} \sum_{j \ge 1} d_{r,j}\, g^j = \sum_{j \ge 1} g^j \sum_{r \ge 1} v_{\alpha,r}\, d_{r,j}.
\end{multline*}
From Lemma 1 and the previous step,
$$
\pi(v_{\alpha,r}) \ge \left\lfloor \dfrac{13r - 2}{14} \right\rfloor, \qquad
\pi(d_{r,j}) \ge \left\lfloor \dfrac{13j - r - 8}{14} \right\rfloor.
$$
Thus,
\begin{multline*}
\pi\left( \sum_{r \ge 1} v_{\alpha,r}\, d_{r,j} \right) \ge \min_{r} \left\{ \pi(v_{\alpha,r}) + \pi(d_{r,j}) \right\} \\
\ge \min_{r} \left\{ \left\lfloor \dfrac{13r - 2}{14} \right\rfloor + \left\lfloor \dfrac{13j - r - 8}{14} \right\rfloor \right\} = \left\lfloor \dfrac{13j - 9}{14} \right\rfloor,
\end{multline*}
since the minimum again occurs at $r = 1$.
Hence,
$$
\pi(u_{\alpha,r}) \ge \left\lfloor \dfrac{13r - 9}{14} \right\rfloor.
$$
This completes the proof of the lemma.
\end{proof}

Now we are ready to prove the conjecture:
\begin{repeatconjecture}{AOBconj}
Let $\alpha \geq 1$, and $p \neq 13$ be a prime $ \geq 5$. Then there exists a constant $k = k(p, \alpha)$ such that for all $N$,
$$
P(p^2 \cdot 13^\alpha N) - \left\{ k - \left(\frac{-3 \cdot 13^\alpha N}{p}\right)p^{-2} \right\} P(13^\alpha N) + p^{-3} P\left(\frac{13^\alpha N}{p^2}\right) \equiv 0 \pmod{13^\alpha},
$$
where $\left(\frac{a}{b}\right)$ is the Jacobi symbol. Equivalently,
\begin{equation}
\left(\cfrac1{\eta(\tau)}\Bigr|U^\alpha\Bigr|T_{p^2}\right)\equiv k\left(\cfrac1{\eta(\tau)}\Bigr|U^\alpha\right)\pmod{13^\alpha}.
\end{equation}
\end{repeatconjecture}

This is equivalent to 
\[
H_{2\alpha-1}
 \;\equiv\;
 \beta_{\alpha}\,L_{2\alpha-1}
 \quad
 \bigl(\bmod\,13^{\,2\alpha-1}\bigr)
 \;\Longrightarrow\;
 u_{\alpha,r}
 \;\equiv\;
 \beta_{\alpha}\,k_{\alpha,r}
 \;\bigl(\bmod\,13^{\,2\alpha-1}\bigr).
\]
and
\[
H_{2\alpha}
 \;\equiv\;
 e_{\alpha}\,L_{2\alpha}
 \quad
 \bigl(\bmod\,13^{\,2\alpha}\bigr)
 \;\Longrightarrow\;
 v_{\alpha,r}
 \;\equiv\;
 e_{\alpha}\,\ell_{\alpha,r}
 \;\bigl(\bmod\,13^{\,2\alpha}\bigr).
\]

\medskip
Let
\[
\mu_{st}^{\alpha}
 \;=\;
 u_{\alpha,s}\,k_{\alpha,t}
 \;-\;
 u_{\alpha,t}\,k_{\alpha,s},
\qquad
\gamma_{st}^{\alpha}
 \;=\;
 v_{\alpha,s}\,\ell_{\alpha,t}
 \;-\;
 v_{\alpha,t}\,\ell_{\alpha,s}.
\]
We want to show
\[
\pi\bigl(\mu_{st}^{\alpha}\bigr)
 \;\ge\;
 2\alpha - 1
 \quad\text{for all }s,t\ge1,
\qquad
\pi\bigl(\gamma_{st}^{\alpha}\bigr)
 \;\ge\;
 2\alpha
 \quad\text{for all }s,t\ge1,
\]
since if we set $s=r$ and $t=1$, then
\[
\mu_{r1}^{\alpha}
 \;=\;
 u_{\alpha,r}\,k_{\alpha,1}
 \;-\;
 u_{\alpha,1}\,k_{\alpha,r}
 \;\equiv\;
 0
 \quad(\bmod\,13^{\,2\alpha-1}),
\]
\[
\gamma_{r1}^{\alpha}
 \;=\;
 v_{\alpha,r}\,\ell_{\alpha,1}
 \;-\;
 v_{\alpha,1}\,\ell_{\alpha,r}
 \;\equiv\;
 0
 \quad(\bmod\,13^{\,2\alpha}),
\]
and
\[
u_{\alpha,r}
 \;\equiv\;
 \frac{u_{\alpha,1}}{k_{\alpha,1}}\,
   k_{\alpha,r}
 \quad(\bmod\,13^{\,2\alpha-1}),
\qquad
v_{\alpha,r}
 \;\equiv\;
 \frac{v_{\alpha,1}}{\ell_{\alpha,1}}\,
   \ell_{\alpha,r}
 \quad(\bmod\,13^{\,2\alpha}).
\]
Since $k_{\alpha,1}\not\equiv 0 \pmod{13}$ and $\ell_{\alpha,1}\not\equiv 0 \pmod{13}$ by Lemma~\ref{AOB_val}, one obtains the claimed congruence.

\medskip

\begin{theorem}\label{thm:MuGammaVal}
We have the following lower bounds on valuations:
\begin{align}
\pi\bigl(\mu^\alpha_{s,t}\bigr) &\ge 2\alpha - 1 
  + \left\lfloor \frac{13(s+t) - 46}{14} \right\rfloor, 
  \quad &&\text{if } s+t > 3, \\[1mm]
\pi\bigl(\mu^\alpha_{s,t}\bigr) &\ge 2\alpha - 1, 
  \quad &&\text{if } s+t = 3, \\[1mm]
\pi\bigl(\gamma^\alpha_{s,t}\bigr) &\ge 2\alpha 
  + \left\lfloor \frac{13(s+t) - 33}{14} \right\rfloor.
\end{align}

\end{theorem}

\begin{proof}
First, note that when $s+t=2$, 
\[
\mu^\alpha_{1,1} \;=\; 0 
\quad\text{and}\quad
\gamma^\alpha_{1,1} \;=\; 0.
\]
So, in what follows, we will assume that $s+t\ge3$.

\medskip
For $\alpha=1$, we have
\[
\mu^1_{s,t} \;=\; u_{1,s}\,k_{1,t} \;-\; u_{1,t}\,k_{1,s}.
\]
Recall the known valuations:
\[
\pi\bigl(u_{j}\bigr)
 \;\ge\;
 \left\lfloor \frac{13\,j - 9}{14}\right\rfloor,
\quad
\pi\bigl(k_{j}\bigr)
 \;\ge\;
 \left\lfloor \frac{13\,j - 9}{14}\right\rfloor.
\]
Hence
\begin{align*}
\pi(\mu^1_{st}) &\geq \min\{\pi(u_{1s}) + \pi(k_{1t}), \pi(u_{1t}) + \pi(k_{1s})\} = \\[2mm]
&= \min\left\{\left\lfloor\frac{13s-9}{14}\right\rfloor + \left\lfloor\frac{13t-9}{14}\right\rfloor, \left\lfloor\frac{13t-9}{14}\right\rfloor + \left\lfloor\frac{13s-9}{14}\right\rfloor\right\} \geq \\[2mm]
&\geq \left\lfloor\frac{13(s+t)-18-13}{14}\right\rfloor = 1 + \left\lfloor\frac{13(s+t)-45}{14}\right\rfloor \geq \\[2mm]
&\geq 1 + \left\lfloor\frac{13(s+t)-46}{14}\right\rfloor \quad \text{if } s+t > 3.
\end{align*}
If $(s + t) = 3$, then either $s=1,t=2$ or $s=2,t=1$, and we get the result by direct calculation.

Assume the statement holds for all $\mu^k_{s,t}$, $k\le\alpha$. Next, consider
\begin{align*}
\gamma_{st}^{\alpha} &= v_{\alpha s}\ell_{\alpha t} - v_{\alpha t}\ell_{\alpha s} = \left(\sum u_{\alpha i}c_{is}\right)\left(\sum k_{\alpha j}c_{jt}\right) - \left(\sum u_{\alpha i}c_{it}\right)\left(\sum k_{\alpha j}c_{js}\right) \\
&= \sum c_{is}c_{jt}(u_{\alpha i}k_{\alpha j} - u_{\alpha j}k_{\alpha i}) = \sum c_{is}c_{jt}\mu_{ij}^{\alpha}
\end{align*}
Then
\begin{align*}
\pi(\gamma_{st}^{\alpha}) &\geq \min_{i,j}\left\{\left\lfloor\frac{13s-i-1}{14}\right\rfloor + \left\lfloor\frac{13t-j-1}{14}\right\rfloor + 2\alpha-1 + \left\lfloor\frac{13(i+j)-46}{14}\right\rfloor\right\} \geq \\
&\geq \min_{i,j}\left\{\left\lfloor\frac{13(s+t)-(i+j)-15}{14}\right\rfloor + 2\alpha-1 + \left\lfloor\frac{13(i+j)-46}{14}\right\rfloor\right\}
\end{align*}
if $i+j>3$ and 
\begin{align*}
\pi(\gamma_{st}^{\alpha}) &\geq \left\lfloor\frac{13(s+t)-3-15}{14}\right\rfloor + 2\alpha-1 = \\
&= 2\alpha + \left\lfloor\frac{13(s+t)-32}{14}\right\rfloor \quad \text{if } i+j = 3
\end{align*}
The total minimum is obtained if $i+j=4$. Therefore,
\[
\pi(\gamma_{st}^{\alpha}) \geq 2\alpha + \left\lfloor\frac{13(s+t)-33}{14}\right\rfloor
\]

To complete our induction, we need to prove the bound for $\mu_{s,t}^{\alpha+1}$. Since
\begin{multline*}
\mu_{st}^{\alpha+1} = u_{\alpha+1,s}k_{\alpha+1,t} - u_{\alpha+1,t}k_{\alpha+1,s} = \left(\sum v_{\alpha i}d_{is}\right)\left(\sum \ell_{\alpha j}d_{jt}\right)\\ -  \left(\sum v_{\alpha i}d_{it}\right)\left(\sum \ell_{\alpha j}d_{js}\right) 
= \sum d_{is}d_{jt}\gamma_{ij}^{\alpha},
\end{multline*}
we have
\begin{align*}
\pi(\mu_{st}^{\alpha+1}) &\geq \min_{i,j}\left\{\left\lfloor\frac{13s-i-8}{14}\right\rfloor + \left\lfloor\frac{13t-j-8}{14}\right\rfloor + 2\alpha + \left\lfloor\frac{13(i+j)-33}{14}\right\rfloor\right\} \geq \\
&\geq \min\left\{\left\lfloor\frac{13(s+t)-(i+j)-29}{14}\right\rfloor + 2\alpha + \left\lfloor\frac{13(i+j)-33}{14}\right\rfloor\right\} = \\
&= \left\lfloor\frac{13(s+t)-3-29}{14}\right\rfloor + 2\alpha + \left\lfloor\frac{13\cdot3-33}{14}\right\rfloor = \\
&= 2\alpha + 1 + \left\lfloor\frac{13(s+t)-46}{14}\right\rfloor
\end{align*}
for $s+t>3$ since the minimum is reached when $i+j=3$. For $s+t=3$, $i+j>3$ we get
\begin{align*}
\pi(\mu_{st}^{\alpha+1}) &\geq \min_{i,j}\left\{\left\lfloor\frac{13s-i-8}{14}\right\rfloor + \left\lfloor\frac{13t-j-8}{14}\right\rfloor + 2\alpha + \left\lfloor\frac{13(i+j)-33}{14}\right\rfloor\right\} \geq \\
&\geq \min\left\{\left\lfloor\frac{13(s+t)-(i+j)-29}{14}\right\rfloor + 2\alpha + \left\lfloor\frac{13(i+j)-33}{14}\right\rfloor\right\} = \\
&= \min\left\{\left\lfloor\frac{10-(i+j)}{14}\right\rfloor + 2\alpha + \left\lfloor\frac{13(i+j)-33}{14}\right\rfloor\right\} = \\
&= 2\alpha + 1.
\end{align*} 
Finally, if $s+t=3$ and $i+j=3$, then
\[
\pi(\mu_{st}^{\alpha+1})\ge 2\alpha+1
\]
by direct calculation.
\end{proof}

\end{document}